\documentclass[a4paper, 12pt] {article}
\usepackage[cp1251]{inputenc}
\usepackage[T1]{fontenc}
\usepackage[english]{babel}
\usepackage{amsfonts}
\usepackage{mathtext}
\usepackage{cite}

\usepackage{graphicx}

\usepackage{verbatim}
\usepackage{amsmath}
\usepackage{amsthm}
\usepackage{amssymb}
\usepackage{delarray}
\usepackage{mathrsfs}
\usepackage{xcolor}
\usepackage{tikz}

\usetikzlibrary{decorations.pathreplacing,matrix,calc,positioning}

\theoremstyle{plain}
\newtheorem{theorem}{\indent Theorem}
\newtheorem{lemma}{\indent Lemma}
\newtheorem{corollary}{\indent Corollary}

\newtheorem{proposition}{\indent Proposition}
\newtheorem{IP}{\indent Inverse problem}

\theoremstyle{remark}

\theoremstyle{plain}

\newtheorem*{Hypothesis*}{\indent Hypothesis}

\theoremstyle{definition}
\newtheorem{remark}{\indent Remark}

\newtheorem*{Question*}{\indent Question}
\newtheorem*{Definition*}{\indent Definition}
\newtheorem*{Notation*}{\indent Notation}

\usepackage{hyperref}

\numberwithin{equation}{section}

\begin{document}
\begin{center}
{\bf\Large Uniform stability of recovering the Sturm--Liouville operator on a star-graph
}
\end{center}

\begin{center}
{\bf\large Maria Kuznetsova\footnote{Saratov State University, Saratov, Russia, e-mail: {\it kuznetsovama@sgu.ru}}}
\end{center}

\noindent{\bf Abstract.}
 In the paper, we study the problem of recovering the Sturm--Liouville operator
on a star-graph from the Weyl vector. It generalizes the problem of recovering
the classical Sturm-Liouville operator on an interval from the Weyl function,
and the problems of recovering from other spectral data can be reduced to this
problem. The uniqueness and the constructive method for solving the problem
under study were previously obtained by V.A. Yurko in the case of a tree
(Inverse Problems, 2005). Here, we prove its uniform stability, which includes
Lipschitz estimates with a constant depending only on the number bounding the
norms of the potentials. Stability results are necessary for justifying the
well-posedness of the problem statement, and they are important for developing
numerical methods. As auxiliary results, we obtain the uniform stability of the
direct problem, as well as the uniform stability of the partial derivatives of
the transmutation operator kernel related to the classical Sturm--Liouville
operator.


\medskip
\noindent {\it Keywords}: inverse spectral problem, Sturm--Liouville operator, Weyl vector, uniform stability, star-graph, transmutation operator kernel.

\medskip
\noindent {\it 2010 Mathematics Subject Classification}: 34A55
\\

\section{Introduction}
This paper is devoted to an inverse spectral problem for the Sturm--Liouville differential operator on a graph.
Differential operators on graphs are actively studied due to their applications to organic chemistry, photonic crystal theory, waveguide theory, and nanotechnology, see, e.g.,~\cite{swen, kuchment, korot}.
A general theory of differential operators on graphs, as well as an extensive bibliography, can be found in the monographs of Yu.V.~Pokorny et al.~\cite{pok}, G.~Berkolaiko and P.~Kuchment~\cite{berk}, and P.~Kurasov~\cite{kur}.

At the beginning of the 21st century, there arose significant interest in inverse spectral problems for differential operators on graphs, see~\cite{kur,avd,umn,A,yur-cycle,cycle,bessel,A-ign,moch,Liu,but,wang,part,arxiv,ign,brown,bel,yurko2005,math,filippo,finco,piv,fam,tree}.
Such problems consist in recovering operators from their spectral characteristics.
The first general results on inverse spectral problems on graphs were obtained for the Sturm--Liouville operators on tree graphs.
A constructive method of solution was first proposed by V.~A.~Yurko~\cite{yurko2005}. 
We also point out the works~\cite{bel,brown}, where the uniqueness of recovering from a larger number of spectral characteristics than in~\cite{yurko2005} was proved.
Later on, there were studied the inverse problems for differential operators on graphs with cycles~\cite{umn,A,yur-cycle,cycle,kur} and on non-compact graphs~\cite{umn,bessel,A-ign,moch}.
Note that a differential operator on a graph can be considered as a special case of the matrix operator (see~\cite{AMP,ipse,xu}), but for recovering the matrix operator more spectral characteristics are required.

Here, we consider the following inverse problem for the Sturm--Liouville operator on a star-graph: recover the potential from $N-1$ Weyl functions, where $N$ is the number of boundary (pendant) vertices.
This problem was first formulated
in the paper~\cite{yurko2005} in the more general case of a tree. It generalizes the inverse problem of recovering the classical Sturm--Liouville operator on an interval from a single Weyl function, see~\cite{freil}. Moreover, the problems of recovering from the spectral data or from several spectra reduce to the problem of recovering from $N-1$ Weyl functions, see~\cite{umn}.

Now, we provide a precise statement of the problem under study. Let $\Gamma$
be a star-graph consisting of $N>1$ edges $\{ e_j \}_{j=1}^N$ of the same length $\pi,$
having one common internal vertex.
The Sturm--Liouville equation on $\Gamma$ reduces to the set of equations
\begin{equation} \label{eq}
-y_j''(x_j) + q_j(x_j) y_j(x_j) = \lambda y_j(x_j), \quad x_j \in (0, \pi),
\quad j = \overline{1, N},
\end{equation}
with the standard matching conditions
\begin{equation} \label{cont}
\begin{split}
y_1(0) = y_j(\pi), \quad j = \overline{2, N},
\\
y'_1(0) = \sum_{j=2}^N y'_j(\pi),
\end{split}
\end{equation}
where $q_j \in L_2(0, \pi),$ $j=\overline{1, N}.$
The vector function $\mathrm{q} = [q_j ]_{j=1}^N$ with the norm
$\| \mathrm{q} \| = \max_{j=\overline{1, N}} \|q_j\|_{L_2(0, \pi)}$ is called potential on the graph.

A vector function $\mathrm{y} = [y_j]_{j=1}^N,$ whose components satisfy system~\eqref{eq}--\eqref{cont}, is called solution of this system.
We introduce the Weyl solution $\Phi_k(x, \lambda) = [ \Phi_{jk}(x, \lambda)]_{j=1}^N$ with the index $k=\overline{2, N}$ as the solution
under the boundary conditions
\begin{equation} \label{Weyl bc}
\Phi_{1k}(\pi, \lambda) = 0, \quad \Phi_{jk}(0, \lambda) = \delta_{jk}, \quad j, k=\overline{2, N},
\end{equation}
where $\delta_{jk}$ is the Kronecker delta. Weyl functions are introduced as
$$M_k(\lambda) = \Phi'_{kk}(0, \lambda), \quad k=\overline{2, N}.$$
The Weyl vector is ${\mathrm M}(\lambda) = [M_k(\lambda)]_{k=2}^N.$
Let us formulate an inverse problem.
\begin{IP} \label{main ip}
Given the Weyl vector ${\mathrm M}(\lambda),$ recover $\mathrm{q}.$
\end{IP}
The uniqueness and the constructive method for Inverse problem~\ref{main ip} were obtained in~\cite{yurko2005}.
Here, we prove its uniform stability. Typically, the stability of inverse spectral problems (see~\cite{freil,sav,borg-compl,sav2}) involves Lipschitz estimates of the form $\| \mathrm{q} - \tilde{\mathrm q}\| \le C\varepsilon,$ where $\mathrm{q}$ and $\tilde{\mathrm q}$ are the coefficients of two different operators, and
$\varepsilon$ is the distance between their spectral characteristics in an appropriate metric. Stability is uniform if the constant $C$ in the estimate is the same for all $\mathrm{q}$ and $\tilde{\mathrm q}$ bounded in norm by a fixed number.

The stability of recovering the Sturm--Liouville differential operators on graphs was studied by Mochizuki~K. and Trooshin~I.~\cite{moch}, Bondarenko~N.P.~\cite{cycle,tree}, Chitorkin~E.E. and Bondaren\-ko~N.P.~\cite{arxiv}. In the paper~\cite{moch}, some estimates were obtained for the scattering problem on the lasso graph with a cycle and an infinite edge, in which, however, spectral characteristics do not participate explicitly. In~\cite{cycle,tree}, the uniform stability of recovering from the characteristic functions was proved  in the case of a graph with a cycle and in the case of a tree with a singular potential.
In the recent work~\cite{arxiv}, the stability of the inverse problem from the spectral data (the spectrum and the weight numbers) was obtained for a real-valued potential on a star graph. Here, we prove the uniform stability of recovering a complex-valued potential from the Weyl vector. This result does not follow from the stability of the inverse problems from the characteristic functions and from the spectral data, due to the possibility of multiple eigenvalues.
On the other hand, the stability of recovering from the characteristic functions can be easily derived from the stability of Inverse problem~\ref{main ip}, see Section~\ref{char}.

Let us provide the main result.
For this purpose, we need the numbers $$\omega_j = \frac12 \int_0^\pi q_j(t)\, dt, \quad j=\overline{1, N}.$$
Let $\mathrm{q}$ and $\mathrm{\tilde q}$ be different potentials on $\Gamma.$ We agree that if a symbol $\alpha$ denotes an object related to the first potential $\mathrm{q},$ then the symbol $\tilde \alpha$ denotes the similar object related to the second potential $\mathrm{\tilde q}.$ For brevity, we put $\hat \alpha = \alpha - \tilde \alpha.$
For an arbitrary $\tau >0,$ we introduce the quantity
\begin{equation*}\label{tau-norm}
\|\mathrm{\hat M}\|_{L_2({\mathbb R}+i \tau)} := \max_{k=\overline{2, N}} \sqrt{\mathop{\int}_{{\mathbb R}+i \tau}\big|\hat M_k(\rho^2)\big|^2 \,d\rho}.
\end{equation*}

\begin{theorem}[stability of the inverse problem] \label{main th}
Let $R>0$ be given and
\begin{equation}
\label{b}
b = \max\left\{ \frac{\ln 4}{2 \pi}, \frac{4 R (\sqrt{\pi} + \sqrt{2\pi} + 24 \pi^2 R \exp(2\pi^\frac32 R))}{\sqrt[N]{8^N + 3^{N-1}}-8} \right\}.\end{equation}
If $\mathrm{q}$ and $\mathrm{\tilde q}$ are potentials on the graph satisfying the conditions
\begin{equation} \label{omega eq}
\hat \omega_j = 0, \; j=\overline{1, N},
\end{equation}
\begin{equation} \label{q bounded}
\| \mathrm{q}\| < R, \quad \| \tilde{\mathrm{q}} \| < R,
\end{equation}
then
\begin{equation*} \label{main ineq}
\| \hat{\mathrm{q}}\| \le A_R \|\hat{\mathrm M}\|_{L_2({\mathbb R}+ib)},
\end{equation*}
where
the constant $A_R > 0$ depends only on $R.$
\end{theorem}
We will also prove the stability of the direct problem: the value $\|\hat{\mathrm M}\|_{L_2({\mathbb R}+ib)}$ can be estimated by  $\| \hat{\mathrm{q}}\|,$ see Theorem~\ref{th dir}.
From this it follows that the quantity $\|\hat{\mathrm M}\|_{L_2({\mathbb R}+ib)}$ in Theorem~\ref{main th} is well-defined.

The proof of Theorem~\ref{main th} utilizes the approach of the paper~\cite{yurko2005}, based on the method of spectral mappings, see~\cite{freil}.
To apply this approach, it is necessary that the parameter $b$ is sufficiently large, and all poles of the Weyl functions lie inside the integration contour $\gamma = ({\mathbb R} + ib)^2.$ For the sake of completeness, we present a specific dependence of $b$ on $R,$ which was absent in the previous works.
To obtain a formula of this dependence, we studied the stability
of the partial derivatives of the TOKs (transmutation operator kernels) of the classical Sturm--Liouville operator, see Appendix~\ref{App}.

The paper is organized as follows. In Section~\ref{aux}, we introduce solutions on the individual edges along with the characteristic functions and obtain necessary estimates for them. Section~\ref{direct problem} is devoted to the stability of the direct problem. In Section~\ref{IP(k)}, we apply the approach from~\cite{yurko2005} for proving the stability. 
Section~\ref{final} contains formulas for reducing Inverse problem~\ref{main ip} to the problem on a smaller subtree, application of which completes the proof of Theorem~\ref{main th}. In Section~\ref{char}, the inverse problem of recovering from the characteristic functions is introduced, and its uniform stability is obtained, see Theorem~\ref{444}.  
In the Appendix,
we prove the uniform stability of the partial derivatives of TOKs with respect to the potential of the classical Sturm--Liouville operator, see Theorem~\ref{trans stab}.

\section{Designations and auxiliary objects}
\label{aux}

For $j=\overline{1, N},$ let $S_j(x, \lambda)$ and $C_j(x, \lambda)$ be the solutions of the $j$-th equation in~\eqref{eq} under the initial conditions
$$S_j(0, \lambda) = C_j'(0, \lambda) =0, \quad S'_j(0, \lambda) = C_j(0, \lambda) =1. $$
Introduce $\rho := \sqrt{\lambda}$ such that $\arg \rho\in [0, \pi).$
Then, we have $\tau := \mathrm{Im}\, \rho \ge 0.$
Here and below, we denote by $A_{R}$ {\it different} positive constants depending on the parameter $R.$
\begin{lemma} \label{S_j lemma}
For $j=\overline{1, N},$ the following formulas hold:
\begin{equation} \label{S_j}
\begin{split}
S_j(\pi, \lambda) = \frac{\sin \rho \pi}{\rho} - \omega_j \frac{\cos \rho\pi}{\rho^2} +\frac{\kappa_{1j}(\rho)}{\rho^2},
\\[2mm]
S'_j(\pi, \lambda) = \cos \rho \pi + \omega_j \frac{\sin \rho \pi}{\rho} + \frac{\kappa_{2j}(\rho)}{\rho},
\\[2mm]
C_j(\pi, \lambda) =\cos \rho \pi + \omega_j \frac{\sin \rho \pi}{\rho} + \frac{\kappa_{3j}(\rho)}{\rho},
\\[2mm]
C'_j(\pi, \lambda) =-\rho\sin \rho \pi + \omega_j \cos \rho \pi - \kappa_{4j}(\rho),
\end{split}
\end{equation}
where $\kappa_{sj}(\rho) = o(e^{\tau \pi}),$ and for every fixed $\tau \ge 0$ and $s=\overline{1, 4},$ $\kappa_{sj} \in L_2(\mathbb{R}+i\tau).$  Moreover, the following statements hold.
\begin{enumerate}
\item
If $\| q_j\|_{L_2(0, \pi)} \le R,$ then
\begin{equation} \label{kappa_sj}
\begin{split}
&|\kappa_{sj}(\rho)| \le \sqrt\pi R B_R e^{\tau \pi}, \quad B_R := \frac{\sqrt 2}{2} + 12R \pi^{\frac32}\exp(2\pi^{\frac32}R),\\
&\|\kappa_{sj}(\rho)\|_{L_2(\mathbb R + i \tau)} \le A_R e^{\tau \pi}.
\end{split}
\end{equation}
\item If $\| q_j\|_{L_2(0, \pi)} \le R$ and $\| \tilde q_j\|_{L_2(0, \pi)} \le R,$ then
\begin{equation} \label{hat kappa} 
|\hat\kappa_{sj}(\rho)| \le A_R \|\hat q_j \|_{L_2(0,\pi)} e^{\tau \pi}, \quad 
\|\hat \kappa_{sj}(\rho)\|_{L_2(\mathbb R + i \tau)} \le A_R \|\hat q_j \|_{L_2(0,\pi)} e^{\tau \pi}.
\end{equation}
\item
If $q_j\in C^{(2)}[0, \pi],$ then
\begin{equation}\label{smooth kappa}
\kappa_{sj}(\rho) = \left\{
\begin{array}{cc}
\displaystyle\Big(\frac{q_j(\pi) + q_j(0)}{4} + \frac{(-1)^s \omega_j^2}{2}\Big) \frac{\sin \rho \pi}{\rho} + O\Big(\frac{e^{\tau \pi}}{\rho^2}\Big), & s=1, 4.\\[3mm]
\displaystyle\Big((-1)^s\frac{q_j(0) - q_j(\pi)}{4} -\frac{ \omega_j^2}{2}\Big) \frac{\cos \rho \pi}{\rho} + O\Big(\frac{e^{\tau \pi}}{\rho^2}\Big), & s=2, 3.
\end{array}\right.
\end{equation}
\end{enumerate}
\end{lemma}
\begin{proof}
Assume $s=1$ for definiteness, the other cases are treated similarly. Formu\-la~\eqref{0}  in Appendix~\ref{App} with $q=q_j$ yields~\eqref{S_j}, where
\begin{equation} \label{kappa1}
\kappa_{1j}(\rho) = \int_0^\pi \dot P(\pi, t) \cos \rho t\, dt.
\end{equation}
By the Riemann--Lebesgue lemma, $\kappa_{1j}(\rho) = o(e^{\tau \pi}).$
Applying Theorem~\ref{trans stab} with $\tilde q = 0,$ we obtain $\| \dot P(\pi, \cdot)\|_{L_2(0, \pi)} \le B_R R.$ Estimating the right-hand side of~\eqref{kappa1} with the Cauchy--Bunyakovsky inequality, we arrive at the first estimate in~\eqref{kappa_sj}. The second estimate in~\eqref{kappa_sj} follows from the fact that~\eqref{kappa1} can be represented as a linear combination of the Fourier transforms.
Applying Theorem~\ref{trans stab} with $q=q_j$ and $\tilde q = \tilde q_j,$ one can prove~\eqref{hat kappa} similarly to~\eqref{kappa_sj}.

If $q_j\in C^{(2)}[0, \pi],$ then $ \dot P(\pi, t) \in C^{(2)}[0, \pi],$  see~\cite{march, freil}. Integrating by parts twice in~\eqref{kappa1}, we get
$$\kappa_{1j}(\rho) = \dot P(\pi, \pi)\frac{\sin \rho \pi}{\rho} + O\Big(\frac{ e^{\tau \pi}}{\rho^2}\Big).$$ Taking into account the formulas of Proposition~\ref{K'(x, x)}, we arrive at~\eqref{smooth kappa}.
\end{proof}

Consider the representations $\Phi_{jk}(x, \lambda) = \alpha_{jk}(\lambda) S_{j}(x, \lambda) + \beta_{jk}(\lambda) C_j(x, \lambda),$ where $j=\overline{1, N}$ and $k=\overline{2, N}.$  Substituting them into conditions~\eqref{cont} and \eqref{Weyl bc}, for each $k,$ we arrive at a system of linear equations with respect to $\{ \alpha_{jk},\beta_{jk}\}_{j=1}^N,$ which yields 
\begin{equation} \label{M_k}
M_k(\lambda) = \alpha_{kk}(\lambda) = -\frac{\Delta_k(\lambda)}{\Delta(\lambda)}, \quad k=\overline{2, N},
\end{equation}
where
\begin{equation} \label{Delta}
\begin{split}
\Delta(\lambda) = C_1(\pi, \lambda)\prod_{l=2}^N S_l(\pi, \lambda) + \sum_{j=2}^N S'_j(\pi, \lambda)\prod_{l\ne j} S_l(\pi, \lambda), \\
\Delta_k(\lambda) = C_k(\pi, \lambda)\Big(C_1(\pi, \lambda)\prod_{l\ne 1,k} S_l(\pi, \lambda) + \sum_{j\ne1,k}S'_j(\pi, \lambda) \prod_{l\ne j,k}S_l(\pi, \lambda)\Big) \\
+ C_k'(\pi, \lambda)\prod_{l\ne k} S_l(\pi, \lambda).
\end{split}
\end{equation}
The function $\Delta(\lambda)$ is the characteristic function of the boundary value problem for~\eqref{eq}, \eqref{cont} with the Dirichlet conditions
$$y_1(\pi) = 0, \quad y_j(0) = 0, \; j = \overline{2, N}.$$
For $k=\overline{2, N},$ the function $\Delta_k(\lambda)$ is the characteristic function of the boundary value problem for~\eqref{eq}, \eqref{cont} with the Dirichlet--Neumann conditions
$$y_1(\pi) = 0, \quad y_j(0) = 0, \; j = \overline{2, N}\setminus\{k\}, \quad y'_k(0) = 0.$$
The functions $\Delta(\lambda)$ and $\Delta_k(\lambda)$ are entire in $\lambda.$
By virtue of formula~\eqref{M_k}, the function $M_k(\lambda)$ is meromorphic, and its poles are zeros of $\Delta(\lambda).$

\begin{lemma} \label{b lemma}
Let $\| \mathrm q\| \le R,$ and let
$b$ be defined in~\eqref{b}.
Then, for $\tau := {\mathrm Im}\,\rho \ge b$ the following inequalities hold:
\begin{equation} \label{S_j ineq}
|S_j(\pi, \lambda)| \ge \frac14\frac{e^{\tau \pi}}{|\rho|}, \quad j=\overline{1, N},
\end{equation}
\begin{equation} \label{Delta ineq}
|\Delta(\lambda)| \ge \frac{N}{4} \Big(\frac{3}{8}\Big)^{N-1}\frac{e^{\tau N \pi}}{ |\rho|^{N-1}}.
\end{equation}
\end{lemma}
\begin{proof}
I.  It follows from~\eqref{S_j} and~\eqref{kappa_sj} that
\begin{equation} \label{L_R}
S_j(\pi, \lambda) = \frac{\sin \rho \pi}{\rho} + \frac{s_j(\rho)}{\rho^2}, \quad |s_j(\rho)|\le e^{\tau \pi} L_R, \quad L_R := R \sqrt{\pi}\Big(\frac12 + B_R\Big).
\end{equation}
For $\tau \ge \ln 4/(2 \pi),$ we have $e^{-2\tau \pi} \le \frac14$ and
$$|\sin \rho \pi| \ge \frac12(e^{\tau \pi} - e^{-\tau \pi}) = \frac{e^{\tau \pi}}{2}(1 - e^{-2\tau \pi}) \ge \frac{3}{8}e^{\tau \pi}.$$
For $\tau \ge 8 L_R,$ we obtain
$$|\rho S_j(\pi, \lambda)| \ge |\sin \rho \pi| - \frac{|s_j(\rho)|}{|\rho|} \ge \frac{3}{8}e^{\tau \pi} - \frac{e^{\tau \pi}L_R}{\tau} \ge \frac14e^{\tau \pi},$$
and~\eqref{S_j ineq} is proved. It remains only to note that $8 L_R \le 8L_R /(\sqrt[N]{8^N + 3^{N-1}}-8) \le b.$

II. For $j=\overline{1, N},$ denote
\begin{equation} \label{ss}
\rho S_j(\pi, \rho)= \sin \rho \pi + a_j(\rho), \quad 
S'_j(\pi, \rho) = \cos \rho \pi + b_j(\rho), \quad
C_j(\pi, \lambda) = \cos \rho \pi + c_j(\rho).
\end{equation}
From~\eqref{L_R} it is clear that
$|a_j(\rho)| \le e^{\tau \pi}L_R/\tau.$
The same inequalities are valid for $|b_j(\rho)|$ and $|c_j(\rho)|.$ Substituting~\eqref{ss} into~\eqref{Delta}, after manipulations, we obtain
\begin{equation} \label{Delta N}
\begin{split}
\Big|\Delta(\lambda) - N \frac{\cos \rho \pi \sin^{N-1} \rho\pi}{\rho^{N-1}}\Big| \le
\frac{N}{|\rho|^{N-1}} \sum_{m=1}^N {N\choose m} \left(\frac{L_R}{\tau}e^{\tau \pi}\right)^m (e^{\tau \pi})^{N-m}\\
= \frac{N}{|\rho|^{N-1}} e^{\tau N \pi} \left( \Big( 1 + \frac{L_R}{\tau}\Big)^N - 1\right).
\end{split}
\end{equation}
It follows from the previous point that for $\tau \ge \ln 4/(2 \pi)$
\begin{equation} \label{Delta est}
\Big|N\frac{\cos \rho \pi \sin^{N-1} \rho\pi}{\rho^{N-1}}\Big| \ge \frac{N}{|\rho|^{N-1}}\frac{3^N}{8^N} e^{\tau N\pi}.
\end{equation}
Moreover, for
$\tau \ge 8L_R /(\sqrt[N]{8^N + 3^{N-1}}-8),$
we have
\begin{equation} \label{NN}
\Big( 1 + \frac{L_R}{\tau}\Big)^N - 1 \le \frac{3^{N-1}}{8^N}.
\end{equation}
Applying~\eqref{Delta N}--\eqref{NN}, we arrive at estimate~\eqref{Delta ineq} for $\tau \ge b = \max \Big\{ \frac{\ln 4}{2 \pi}, \frac{8L_R}{\sqrt[N]{8^N + 3^{N-1}}-8} \Big\}:$
\begin{equation*}
\begin{split}
|\Delta(\lambda)| \ge \frac{N}{|\rho|^{N-1}}\frac{3^N}{8^N} e^{\tau N\pi} - \frac{N}{|\rho|^{N-1}} \left( \Big( 1 + \frac{L_R}{\tau}\Big)^N - 1\right) e^{\tau N\pi} \\ \ge\frac23 \frac{N}{|\rho|^{N-1}} \frac{3^N}{8^N} e^{\tau N\pi} =
\frac{N}{4|\rho|^{N-1}} \Big(\frac{3}{8}\Big)^{N-1} e^{\tau N \pi}.
\end{split}
\end{equation*}
\end{proof}

\section{Stability of the direct problem} \label{direct problem}

In this section, we study the dependence of the Weyl vector $\mathrm{M}(\lambda)$
on the potential ${\mathrm q}$ and prove that this dependence is Lipschitz continuous. From here on, we assume that $\mathrm q$ and $\tilde {\mathrm q}$ satisfy the conditions of Theorem~\ref{main th}.

\begin{theorem}[stability of the direct problem] \label{th dir}
Under the conditions of Theorem~\ref{main th}, the inequality
$$\|\mathrm{\hat M}\|_{L_2({\mathbb R}+ib)} \le A_{R} \| \mathrm{\hat q}\|$$
holds. 
\end{theorem}
 \begin{proof}
 I. Let $\mathrm{Im}\,\rho = b.$ From formulas~\eqref{S_j} and~\eqref{kappa_sj}, it follows that for $j=\overline{1, N}$ 
\begin{equation} \label{CS} 
\begin{split} 
|S_j(\pi, \rho^2)| \le \frac{A_R}{|\rho |}, \quad |S'_j(\pi, \rho^2)| \le A_R, \\[2mm] 
|C_j(\pi, \rho^2)| \le A_R, \quad |C'_j(\pi, \rho^2)| \le A_R |\rho|. 
\end{split}
\end{equation}
Estimating each term in the formula for $\Delta_k$ in~\eqref{Delta},  we easily obtain
\begin{equation}|\rho^{N-2}\Delta_k(\rho^2)| \le A_R, \quad k=\overline{2, N}.
\label{Delta_k est}
\end{equation}
Note that all estimates in this point are also valid for the objects related to $\tilde{\mathrm q}.$

II. Applying formulas~\eqref{S_j} and~\eqref{hat kappa}, by~\eqref{omega eq}, for $j=\overline{1, N}$ we have
\begin{equation} \label{kappa CS}
\begin{split}
\|\rho^2 \hat S_j(\pi, \rho^2)\|_{L_2(\mathbb R + ib)} \le A_R \| \mathrm{\hat q}\|, \quad \|\rho \hat S'_j(\pi, \rho^2)\|_{L_2(\mathbb R + ib)} \le A_R \| \mathrm{\hat q}\|, \\[2mm] 
\|\rho \hat C_j(\pi, \rho^2)\|_{L_2(\mathbb R + ib)} \le A_R \| \mathrm{\hat q}\|, \quad \|\hat C'_j(\pi, \rho^2)\|_{L_2(\mathbb R + ib)} \le A_R \| \mathrm{\hat q}\|. 
\end{split} 
\end{equation} 
Applying~\eqref{Delta} and~\eqref{CS}, we get 
$$|\hat \Delta(\rho^2)| \le \frac{A_R}{|\rho|^{N-1}}|\hat C_1(\pi, \rho^2)| + \frac{A_R}{|\rho|^{N-2}} \sum_{j=2}^N |\hat S_j(\pi, \rho^2)| + \frac{A_R}{|\rho|^{N-1}} \sum_{j=2}^N |\hat S'_j(\pi, \rho^2)|.$$
By~\eqref{kappa CS}, this implies
\begin{equation*}
\| \rho^{N}\hat \Delta(\rho^2)\|_{L_2(\mathbb R + ib)} \le A_R \| \mathrm{\hat q}\|.
\end{equation*}
Similarly,
\begin{equation*}
\| \rho^{N-1}\hat \Delta_k(\rho^2)\|_{L_2(\mathbb R + ib)} \le A_R \| \mathrm{\hat q}\|, \quad k=\overline{2, N}. 
\end{equation*}

III.  By formula~\eqref{M_k}, we have
\begin{equation*}
\begin{split}
\big|\hat M_k(\rho^2)\big| \le \left| \frac{\hat \Delta_k(\rho^2)}{\Delta(\rho^2)}\right| +\left| \frac{\Delta_k(\rho^2)\hat \Delta(\rho^2)}{\Delta(\rho^2)\tilde\Delta(\rho^2)}\right| \\
\overset{\eqref{Delta ineq}}{\le} A_R \big|\rho^{N-1}\hat\Delta_k(\rho^2)\big|
+A_R \big|\rho^{N-2}\Delta_k(\rho^2)\big|\big|\rho^{N}\hat\Delta(\rho^2)\big|.
\end{split}
\end{equation*}
Taking into account the inequalities obtained in points I and II, we arrive at the estimates
$$\big\|\hat M_k(\rho^2)\big\|_{L_2({\mathbb R}+ ib)} \le A_R \|\hat{\mathrm q}\|, \quad k=\overline{2, N}, $$
which are equivalent to the statement of the theorem.
 \end{proof}
  Note that the estimate obtained in Theorem~\ref{th dir} is the opposite of the estimate in Theorem~\ref{main th}, which we want to prove. In its following proof, we will apply the result of Theorem~\ref{th dir}.
We will also need the following lemma.
\begin{lemma} \label{G_delta}
Introduce a class of potentials on $\Gamma$
\begin{equation*}
C^{(2)}_0(\Gamma) := \Big\{ {\mathrm q} = [q_j]_{j=1}^N \colon q_j \in C^{(2)}[0, \pi], \; q_j(0) = q_j(\pi) =0,\ j=\overline{1, N}\Big\}.
\end{equation*}
Let the potentials $\mathrm q, \tilde{\mathrm q} \in C^{(2)}_0(\Gamma)$ satisfy conditions~\eqref{omega eq} and~\eqref{q bounded}.
Then,
$$\hat M_k(\rho^2) = O\left(\frac{1}{\rho^2}\right), \quad \rho \in {\mathbb R}+i b,
\quad k=\overline{2, N}.$$
\end{lemma}
The lemma is proved similarly to Theorem~\ref{th dir}, but instead of formulas~\eqref{hat kappa} one should use formulas~\eqref{smooth kappa}.

\section{Auxiliary inverse problems}
\label{IP(k)}

In this section, we assume $k = \overline{2, N}.$ Consider an inverse problem on the edge $e_k.$
\begin{IP} \label{aux}
Given $M_k(\lambda),$ recover $q_k.$
\end{IP}
The unique solvability of the problem was proved in~\cite{yurko2005}, where it was an auxiliary step in the process of recovering the potential on the whole graph. 
In this section, we prove the uniform stability of Inverse problem~\ref{aux}, see estimates~\eqref{q_k stab}.

For $x \in (0, \pi)$ and $\lambda, \mu \in \mathbb C,$ we introduce 
$$\tilde D_k(x, \lambda, \mu) = \int_0^x \tilde S_k(t, \lambda) \tilde S_k(t, \mu)\, dt = \frac{\langle \tilde S_k(x, \lambda), \tilde S_k(x, \mu)\rangle}{\lambda - \mu},$$
where $\langle \phi, \psi\rangle := \phi \psi' - \phi' \psi$ is the Wronski determinant.
Proceeding similarly to the proof of Lemma~1.6.2 from~\cite{freil}, one can obtain that
\begin{equation} \label{D}
|D_k(x, \rho^2, \theta^2)| \le \frac{A_R e^{(\mathrm{Im} \, \rho + \mathrm{Im} \, \theta)\pi}}{|\rho| |\theta|(|\rho - \theta| + 1)}, \quad \mathrm{Im} \, \rho, \mathrm{Im} \, \theta \ge b.
\end{equation}
Denote
$$\gamma = \{ \rho^2 \colon \rho \in \mathbb R + ib \}, \quad
\gamma_m = \left\{ \rho^2 \colon \rho \in \mathbb R + ib, \; |\rho|<m+\frac14 \right\}, \; m \in \mathbb N.$$
It is easy to see that $\gamma$ is a parabola: $\gamma = \{x + iy \colon x = \frac{y^2}{4b^2} - b^2\}.$  When calculating contour integrals, we assume that $\gamma$ and $\gamma_m$ are counterclockwise, see Figure~\ref{fig:contour}.
Then, $\mathrm{int}\,\gamma = \{ \rho^2 \colon \mathrm{Im}\,\rho \in [0, b)\}.$
  
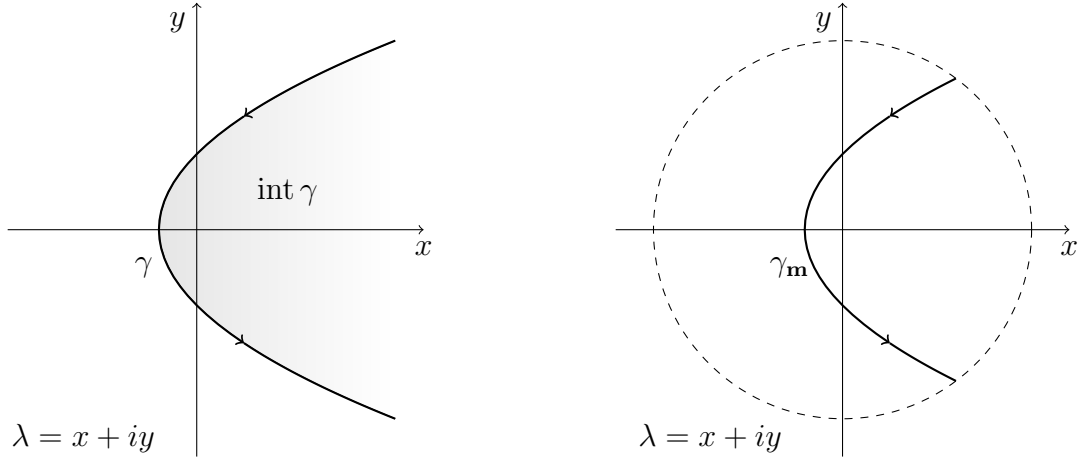
\begin{figure}[h!]
\begin{minipage}{0.5\linewidth}
\centering
\begin{tikzpicture}
\draw[domain=-2.5:2.5,left color=gray!20!white, right color=white]  plot (\x*\x/2 - 0.5,\x);
  \draw[thin,->] (-2.5, 0) -- (3, 0);
    \draw[thin,->] (0, -3) -- (0, 3);
  \draw[thick,domain=-2.5:-1.5] plot (\x*\x/2 - 0.5,\x);
  \draw[thick,domain=-1.5:1.5,<-] plot (\x*\x/2 - 0.5,\x);
  \draw[thick,domain=1.5:2.5,<-] plot (\x*\x/2-0.5,\x);
\node at (1.2, 0.5) {$\mathrm{int}\,\gamma$};
\node at (-1.5, -2.75) {$\lambda = x + i y$};
  \node  at (3, -0.25) {$x$};
    \node  at (-0.25, 2.75) {$y$};
    \node at (-0.7, -0.5) {$\bf \gamma$};
\end{tikzpicture}
\end{minipage}
\begin{minipage}{0.47\linewidth}
\centering
\begin{tikzpicture}
  \draw[thin,->] (-3, 0) -- (3, 0);
    \draw[thin,->] (0, -3) -- (0, 3);
  \draw[thick,domain=-2.0:-1.5] plot (\x*\x/2 - 0.5,\x);
  \draw[thick,domain=-1.5:1.5,<-] plot (\x*\x/2 - 0.5,\x);
  \draw[thick,domain=1.5:2.0,<-] plot (\x*\x/2-0.5,\x);
  \node  at (3, -0.25) {$x$};
    \node  at (-0.25, 2.75) {$y$};
    \node at (-0.7, -0.5) {$\bf \gamma_m$};
\node at (-1.75, -2.75) {$\lambda = x + i y$};
 \draw [dashed] (0, 0) circle (2.5);
\end{tikzpicture}
\end{minipage}
\caption{The contours $\gamma$ and $\gamma_m$}
\label{fig:contour}
\end{figure}

By~\eqref{M_k} and~\eqref{Delta ineq}, the poles of the function $M_k(\lambda)$ lie inside $\gamma.$
Applying the method of spectral mappings analogously to~\cite[\S1.6.1]{freil},
for each fixed $x\in(0, \pi),$ we obtain the {\it main equation} of Inverse problem~\ref{aux}:
\begin{equation} \label{main eq1}
\tilde S_k(x, \lambda) = S_k(x, \lambda) - \frac{1}{2\pi i} \mathop{\int}_{\gamma}  {\tilde D}_k(x, \lambda, \mu) \hat M_k(\mu) S_k(x, \mu) \, d\mu.
\end{equation}
Equations of such type are the key step for proving the uniqueness of solution of inverse spectral problems, see~\cite{freil}.
We use~\eqref{main eq1} to obtain a formula relating the components $\hat q_k$ and $\hat M_k.$

\begin{theorem} For a.e. $x \in (0, \pi),$ the following formula is valid:
\begin{equation} \label{qq}
\hat q_k(x) = \frac{1}{\pi i} \mathop{l.i.m.}_{m \to \infty} \mathop{\int}_{\gamma_m} (\tilde S_k(x, \mu) S_k(x, \mu))' \hat M_k(\mu) \, d\mu.
\end{equation}
\end{theorem}
\begin{proof}
I. First, suppose that $\mathrm{q}, \mathrm{\tilde q} \in C_0^{(2)}(\Gamma).$ By Lemma~\ref{G_delta}, we have $\hat M_k(\rho^2) \in L(\gamma).$ Differentiating with respect to $x$ twice both sides of~\eqref{main eq1}, analogously to the proof of~\cite[lemma~1.6.5]{freil}, we obtain the formula
\begin{equation*} \label{C formula}
\hat q_k(x) = \frac{1}{\pi i} \mathop{\int}_{\gamma} (\tilde S_k(x, \mu) S_k(x, \mu))' \hat M_k(\mu) \, d\mu, \quad x \in (0, \pi).
\end{equation*}
In this case, differentiation under the integral signs is valid, because by virtue of~\eqref{CS}, \eqref{D}, and the property $\hat M_k(\rho^2) \in L(\gamma),$ the integrals obtained after differentiation converge absolutely and uniformly in $x,$ see~\cite[\S7.1.5]{podkorytov}.

 II. Consider the general case of potentials $\mathrm{q}$ and $\mathrm{\tilde q}$ with the components from $L_2(0, \pi).$
Let $\{ {\mathrm p}^n\}_{n=1}^\infty$ be a sequence of potentials $\mathrm{p}^n =: [p_j^n(x)]_{j=1}^N$ satisfying the conditions
\begin{equation} \label{p to q}
\mathrm{p}^n \in C^{(2)}_0(\Gamma),\quad
\| \mathrm{p}^n - \mathrm{q}\| \to 0, \; n \to \infty.
\end{equation}
As the components of the sequence members, one can take
partial sums of the Fourier sine series:
\begin{equation} \label{fourier}
p^n_j(x) = \sum_{m=1}^n \beta_j^m \sin mx , \quad \beta_j^m := \frac{2}{\pi} \int_0^\pi \sin mt \, q_j(t)\,dt, \quad j = \overline{1, N},
\end{equation}
then conditions~\eqref{p to q}  are satisfied.
Let $\{ \mathrm{ \tilde p}^n\}_{n=1}^\infty$~ be a sequence constructed in the same way for $\tilde{\mathrm q}.$
We can immediately assume that
$$v_{j}^n := \int_0^\pi \big({p}^n_j(x) - \tilde{p}^n_j(x)\big) dx = 0, \quad j = \overline{1, N}, \; n \ge 1.$$
Otherwise, instead of~\eqref{fourier}, one should take
\begin{equation*}
p^n_j(x) = \sum_{m=1}^n \beta_j^m \sin mx - \frac{ v_{j}^n}{2} \sin x, \quad j = \overline{1, N}.
\end{equation*}
Due to~\eqref{omega eq} and~\eqref{p to q}, we have $\{ v_j^n\}_{n=1}^\infty \to 0,$ and the condition $\| \mathrm{p}^n - \mathrm{q}\| \to 0$ is not violated.

Let $n\ge1$ and ${\mathrm M}^n = [M_s^n(\lambda)]_{s=2}^N$ be the Weyl vector
of the problem~\eqref{eq}--\eqref{cont} with the potential $\mathrm{p}^n.$
Introduce
the solution of the Cauchy problem
$$-(\varphi^n_k)'' + p^n_k(x)\varphi^n_k = \lambda \varphi^n_k, \; x \in (0, \pi),
\quad \varphi^n_k(0) = 0, \; (\varphi^n_k)'(0) = 1.$$
We also define similar objects for the potential $\tilde{\mathrm{p}}^n.$ Inequalities~\eqref{q bounded} together with the properties $\| \mathrm{p}^n - \mathrm{q}\| \to 0$ and
$\| \tilde{\mathrm{p}}^n - \tilde{\mathrm{q}}\| \to 0$   yield $\| \mathrm{p}^n\| < R$ and $\| \tilde{\mathrm{p}}^n\| < R$ for sufficiently large $n >K.$ Then, for $n > K,$ the potentials $\mathrm{q}=\mathrm{p}^n$ and $\tilde{\mathrm q}=\tilde{\mathrm p}^n$ satisfy the conditions of Theorem~\ref{main th} and belong to the class $C^{(2)}_0(\Gamma).$
We have the formula
\begin{equation} \label{pred}
\hat{p}^n_k(x) = \frac{1}{\pi i} \mathop{\int}_{\gamma} (\tilde \varphi^n_k(x, \mu) \varphi^n_k(x, \mu))' \hat{M}^n_k(\mu) \, d\mu, \quad n > K,
\end{equation}
which was proved in point I.
As $n \to \infty$, the left-hand side of the formula tends to $\hat q_k(x)$ in the $L_2$-norm. We will prove that the right-hand side of the formula tends to the right-hand side of~\eqref{qq} in the same norm.

III. Denote the right-hand side of formula~\eqref{qq} by
$$I(S_k, \tilde S_k, \hat M_k; x) := \frac{1}{\pi i} \mathop{l.i.m.}_{m \to \infty} \mathop{\int}_{\gamma_m} (\tilde S_k(x, \mu) S_k(x, \mu))' \hat M_k(\mu) \, d\mu.$$
It follows from Lemma~\ref{S_j lemma} that for $x \in [0, \pi]$
\begin{equation*} \label{SS'}
\rho\big(S_k(x, \rho^2) \tilde S_k(x, \rho^2)\big)' = \sin 2 \rho x + \kappa_{k}(x, \rho), \quad \| \kappa_k(x, \cdot)\|_{L_2({\mathbb R}+ib)} \le A_R.
\end{equation*}
Using this representation, after changing the integration variable $\mu = \rho^2,$ we have
\begin{equation}
I(S_k, \tilde S_k, \hat M_k; x) =    \frac{2}{\pi i} \mathop{l.i.m.}_{m \to \infty}  \int_{-m+ib}^{m+ib} \sin 2 \rho x \hat M_k(\rho^2) \,d\rho + \frac{2}{\pi i}\mathop{\int}_{\mathbb{R}+ib} \kappa_k(x, \rho)\hat M_k(\rho^2) \,d\rho.
\label{I_M}
\end{equation}
Herein, the first term is well-defined as an $L_2$-function of $x,$ being a linear combination of the Fourier transforms of $L_2$-functions:
\begin{equation*}\begin{split}
\mathop{l.i.m.}_{m \to \infty}  \int_{-m+ib}^{m+ib} \sin 2 \rho x \hat M_k(\rho^2) \,d\rho = \mathop{l.i.m.}_{m \to \infty} \int_{-m}^m \frac{e^{2i z x} e^{-2bx} - e^{-2izx}e^{2bx}}{2i} \hat M_k((z + ib)^2)\,dz\\
= \frac{e^{-2bx}}{2i}\mathop{l.i.m.}_{m \to \infty} \int_{-m}^m e^{-i \sigma x} \hat M_k\left(\Big(-\frac\sigma2 + ib\Big)^2\right)\,d\sigma- \frac{e^{2bx}}{2i}\mathop{l.i.m.}_{m \to \infty} \int_{-m}^m e^{-i\sigma x}\hat M_k\left(\Big(\frac\sigma2 + ib\Big)^2\right)\,d\sigma.
\end{split}
\end{equation*}
The first term in~\eqref{I_M} does not exceed $A_R \| \hat M_k(\rho^2) \|_{L_2(\mathbb R + ib)}$ in the $L_2$-norm by the properties of the Fourier transform.
By the Cauchy--Bunyakovsky inequality, the integral in the second term in~\eqref{I_M} converges uniformly and absolutely, and its modulus is bounded by the number $A_R \| \hat M_k(\rho^2) \|_{L_2(\mathbb R + ib)}.$
Thus, $I(S_k, \tilde S_k, \hat M_k; \cdot)$ is well-defined as an $L_2$-function, and
\begin{equation} \label{I M_k}
\| I(S_k, \tilde S_k, \hat M_k; \cdot) \|_{L_2(0, \pi)} \le A_R \| \hat M_k(\rho^2) \|_{L_2(\mathbb R + ib)}.
\end{equation}

We write the right-hand side of the equality~\eqref{pred} as $I(\varphi_k^n, \tilde\varphi_k^n, \hat M^n_k; x).$ Then,
\begin{equation} \begin{split}
I(S_k, \tilde S_k, \hat M_k; x) - I(\varphi_k^n, \tilde\varphi_k^n, \hat M^n_k; x) &= H_n(x) +  I(S_k, \tilde S_k,[M^n_k - {M_k}] + [\tilde M_k - \tilde M^n_k]; x), \\
H_n(x) := \frac{1}{\pi i} \mathop{l.i.m.}_{m \to \infty} &\mathop{\int}_{\gamma_m} [\tilde S_k(x, \mu) S_k(x, \mu)-\tilde \varphi^n_k(x, \mu) \varphi^n_k(x, \mu) ]'\hat M_k(\mu)\,d\mu.
\end{split} \label{pp2}
\end{equation}
Since $\varphi_k^n(x, \lambda)$ denotes the solution corresponding to $p_k^n$, which is analogous to the solution $S_k(x, \lambda)$ corresponding to $q_k$, we can use the formulas of Lemma~\ref{S_j lemma} for it. Applying the corresponding formulas from~\eqref{S_j},~\eqref{hat kappa}, and~\eqref{CS}, we obtain the estimate
$$\big|[\tilde S_k(x, \rho^2) S_k(x, \rho^2)-\tilde \varphi^n_k(x, \rho^2) \varphi^n_k(x, \rho^2)]'\big| \le A_R \rho^{-2} (\|\mathrm{p}^n - \mathrm{q} \| + \|\tilde{\mathrm{p}}^n - \tilde{\mathrm{q}} \|).$$
Using the Cauchy--Bunyakovsky inequality and this estimate, by Theorem~\ref{th dir} we arrive at
$$|H_n(x)|\le A_R (\|\mathrm{p}^n - \mathrm{q} \| + \|\tilde{\mathrm{p}}^n - \tilde{\mathrm{q}} \|)\big\| \hat M_k(\rho^2) \big\|_{L_2(\mathbb R + ib)} \le A_R (\|\mathrm{p}^n - \mathrm{q} \| + \|\tilde{\mathrm{p}}^n - \tilde{\mathrm{q}} \|).$$
Taking into account~\eqref{I M_k},~\eqref{pp2}, and the result of Theorem~\ref{th dir}, we obtain
\begin{equation*}
\begin{split}
\| I(S_k, \tilde S_k, \hat M_k; \cdot) - I(\varphi_k^n, \tilde\varphi_k^n, \hat M^n_k; \cdot)\|_{L_2(0, \pi)} \\
\le  A_R(\| M^n_k  -  M_k\|_{L_2(\mathbb R + ib)}  +\| \tilde M^n_k  -  \tilde M_k \|_{L_2(\mathbb R + ib)})
+A_R (\|\mathrm{p}^n - \mathrm{q} \| + \|\tilde{\mathrm{p}}^n - \tilde{\mathrm{q}} \|) \\
\le  
A_R (\|\mathrm{p}^n - \mathrm{q} \| + \|\tilde{\mathrm{p}}^n - \tilde{\mathrm{q}} \|) \to 0, \quad n \to \infty.
\end{split}
\end{equation*}
Thus, the limit of $\big\{ I(\varphi_k^n, \tilde\varphi_k^n, \hat M^n_k; \cdot) \big\}_{n=1}^\infty$ in $L_2(0, \pi)$ as $n \to \infty$ equals the right-hand side of~\eqref{qq}. On the other hand, by~\eqref{pred}, this limit equals $\hat q_k.$ Formula~\eqref{qq} is proved.
\end{proof}
From the proof, it is clear that the right-hand side of formula~\eqref{qq} is well-defined as an $L_2$-function. Moreover, from~\eqref{qq} and~\eqref{I M_k} we obtain a corollary.

\begin{corollary} For $k=\overline{2, N},$ the following estimates hold:
\begin{equation} \label{q_k stab}
\| \hat q_k \|_{L_2(0, \pi)} \le A_R \| \hat M_k(\rho^2)\|_{L_2({\mathbb R}+ib)}.
\end{equation}
\end{corollary}

\label{IP(k)}

\section{Estimate on the edge $e_1.$ Proof of Theorem~\ref{main th}}
\label{final}

In the previous section, we considered the edges $e_k$ with the number $k=\overline{2, N}$ and obtained necessary estimates~\eqref{q_k stab}.
For the edge $e_1,$ the situation is different: the corresponding Weyl function $M_1$ is absent in the input data of Inverse problem~\ref{main ip}.
We define it as follows:
\begin{equation}\label{M_1}
M_1(\lambda) = -\frac{C_1(\pi, \lambda)}{S_1(\pi, \lambda)}.
\end{equation}
The function $M_1$ can be interpreted as the Weyl function for a tree consisting of a single edge $e_1.$
It is calculated from the Weyl vector and the objects corresponding to the edges with the number $k=\overline{2, N},$ see formula~\eqref{M_1 comp} below. This corresponds to the procedure of  ``cutting" the boundary edges proposed in~\cite{yurko2005}, after which Inverse problem~\ref{main ip} is reduced to the same inverse problem on a smaller subtree.

Proceeding similarly to the proof of Theorem~\ref{th dir}, we obtain that
$$\| \hat M_1(\rho^2)\|_{L_2(\mathbb R + ib)} \le A_R \| q_1\|_{L_2(0, \pi)}. $$
Therefore, $\hat M_1(\rho^2) \in L_2(\mathbb R + ib).$
The assertion of Lemma~\ref{G_delta} is also true for $k=1.$
From estimate~\eqref{S_j ineq} and formula~\eqref{M_1} it follows that the poles of the meromorphic function $M_1(\lambda)$ lie inside the contour~$\gamma.$ This means that all arguments of Section~\ref{IP(k)} are applicable for $k=1.$
As a result, estimate~\eqref{q_k stab} holds for $k=1$ as well.

We obtain a formula for calculating $M_1$ from the Weyl vector $\mathrm{M}(\lambda)=[M_k(\lambda)]_{k=2}^N.$
Let us consider the Weyl solution $\Phi_N(x, \lambda) = [\Phi_{jN}(x, \lambda)]_{j=1}^N.$
From boundary conditions~\eqref{Weyl bc} and the fact that $\{ S_j, C_j \}$ is a fundamental system of solutions on each edge $e_j,$ it follows that
\begin{equation*}
\begin{split}
 \Phi_{1N}(x, \lambda) = \beta_{1N}(\lambda)(M_1(\lambda) S_1(x, \lambda)  + C_1(x, \lambda)), \\
\Phi_{kN}(x, \lambda) = \alpha_{kN}(\lambda) S_k(x, \lambda), \quad k=\overline{2, N-1},\\
\Phi_{NN}(x, \lambda) = M_N(\lambda) S_N(x, \lambda)+C_N(x, \lambda).
\end{split}
\end{equation*}
Substituting these representations into matching conditions~\eqref{cont}, we have
\begin{equation*}
\begin{split}
\beta_{1N}(\lambda) = \alpha_{2N}(\lambda)S_2(\pi, \lambda) = \ldots =  \alpha_{N-1,N}(\lambda)S_{N-1}(\pi, \lambda) = M_N(\lambda) S_N(\pi, \lambda)+C_N(\pi, \lambda),\\
\beta_{1M}(\lambda) M_1(\lambda) = \sum_{k=2}^{N-1} \alpha_{kN}(\lambda)S'_k(\pi, \lambda) + M_N(\lambda) S'_N(\pi, \lambda)+C'_N(\pi, \lambda).
\end{split}
\end{equation*}
From here we obtain
$$M_1(\lambda) = \sum_{j=2}^{N-1} \frac{S'_j(\pi, \lambda)}{S_j(\pi, \lambda)} + \frac{M_N(\lambda) S'_N(\pi, \lambda)+C'_N(\pi, \lambda)}{M_N(\lambda) S_N(\pi, \lambda)+C_N(\pi, \lambda)}.$$
Adding and subtracting $S'_N(\pi, \lambda)/S_N(\pi, \lambda),$  taking into account the identity $S_N C_N' - S_N' C_N \equiv -1,$ after manipulations we arrive at the formula
\begin{equation} \label{M_1 comp}
M_1(\lambda) = \sum_{j=2}^N \frac{S'_j(\pi, \lambda)}{S_j(\pi, \lambda)} - \frac{1}{S_N(\pi, \lambda)(S_N(\pi, \lambda)M_N(\lambda) + C_N(\pi, \lambda))}.\end{equation}

\begin{lemma} Let us denote $F_k(\lambda) := S_k(\pi, \lambda)(S_k(\pi, \lambda)M_k(\lambda) + C_k(\pi, \lambda)),$ $k=\overline{2, N}.$ Then,
\begin{equation} \label{F_k}
F_k(\lambda) = \frac{S_1(\pi, \lambda) \ldots S_N(\pi, \lambda)}{\Delta(\lambda)}.
\end{equation}
\end{lemma}
\begin{proof}
For brevity, we omit the argument  $\lambda$ in the functions $\Delta$ and $\Delta_k,$ as well as the pair of arguments $(\pi, \lambda)$ in $S_j,$ $C_j,$ and in their derivatives.
Applying~\eqref{M_k} and \eqref{Delta}, we have
\begin{equation*}
\frac{F_k}{S_k} = \frac{1}{\Delta}\left[ C_k \Big(C_1 \prod_{l=2}^N S_l + \sum_{j=2}^N S_j' \prod_{l \ne j} S_l\Big) - S_k C_k \Big(C_1 \prod_{l\ne 1, k} S_l + \sum_{j\ne1,k} S_j' \prod_{l\ne j,k} S_l\Big) - S_k C_k' \prod_{l\ne k}S_l\right].
\end{equation*}
Expanding the brackets and canceling out terms with opposite signs, by virtue of the identity $C_k S_k' - C_k'S_k \equiv 1,$ we obtain
$$
\frac{F_k}{ S_k} = \frac{1}{\Delta}\left[C_k S_k' \prod_{l \ne k} S_l - S_k C_k' \prod_{l\ne k}S_l\right] = \frac{1}{\Delta} \prod_{l \ne k} S_l.
$$
\end{proof}
From the statement of the lemma, it is clear that the last term $-1/F_N(\lambda)$ in formula~\eqref{M_1 comp} can be replaced by any other term $-1/F_k(\lambda),$ $k=\overline{2, N-1}.$ Thus, we obtain the same formula~\eqref{M_1 comp} regardless of the index $k=\overline{2, N}$ of the Weyl solution $\Phi_k(x, \lambda)$ under consideration.

\begin{proof}[Proof of Theorem~\ref{main th}]
In view of estimates~\eqref{q_k stab} for $k=\overline{1, N},$ it remains to prove that
\begin{equation} \label{M_1 stab}
\big\| \hat M_1(\rho^2)\big\|_{L_2({\mathbb R}+ib)} \le A_R \big\| \hat {\mathrm M}\big\|_{L_2({\mathbb R}+ib)}.
\end{equation}
Denote
$$I_j(\lambda) = \frac{S'_j(\pi, \lambda)}{S_j(\pi, \lambda)}, \quad j=\overline{2, N}.$$
From formula~\eqref{M_1 comp} it follows that
\begin{equation} \label{|M_1|}
\big|\hat M_1(\rho^2)\big| \le \sum_{j=2}^N \big|\hat I_j(\rho^2)\big| + \left|\frac{\hat F_N(\rho^2)}{F_N(\rho^2) \tilde F_N(\rho^2)}\right|.
\end{equation}
Let us estimate the last term on the right-hand side of~\eqref{|M_1|}. Similarly to~\eqref{Delta_k est}, we obtain
\begin{equation} \label{D est}
|\rho^{N-1} \Delta(\rho^2)| \le A_R, \quad \rho \in {\mathbb R} + ib.
\end{equation}
Due to~\eqref{S_j ineq},~\eqref{F_k}, and~\eqref{D est}, we have
\begin{equation}
|F_N(\rho^2)|^{-1} \le A_R|\rho|, \quad |\tilde F_N(\rho^2)|^{-1} \le A_R|\rho|, \quad \rho \in {\mathbb R}+ib.
\label{F_N^-1}
\end{equation}
Furthermore,
$$|\hat F_N(\rho^2)| \le \big(|S_N M_N| + |\tilde S_N \tilde M_N| +|C_N|\big)|\hat S_N| + |S_N \tilde S_N||\hat M_N| + |\tilde S_N||\hat C_N|. $$
Using~\eqref{CS},~\eqref{kappa CS}, and the estimate $|M_N(\rho^2)| \le A_R |\rho|$ for $\rho \in {\mathbb R}+ib,$ we arrive at
$$
\big\| \rho^2 \hat F_N(\rho^2) \big\|_{L_2({\mathbb R}+ib)} \le A_R \big(\| \hat q_N\|_{L_2(0, \pi)} + \| \hat M_N(\rho^2)\|_{L_2({\mathbb R}+ib)}\big)
\overset{\eqref{q_k stab}}{\le} A_R \big\| \hat M_N(\rho^2)\big\|_{L_2({\mathbb R}+ib)}.
$$
Therefore, due to~\eqref{F_N^-1},
\begin{equation}
\left\| \frac{\hat F_N(\rho^2)}{ F_N(\rho^2) \tilde F_N(\rho^2)}\right\|_{L_2({\mathbb R}+ ib)} \le A_R\big\| \hat{\mathrm M}\big\|_{L_2({\mathbb R}+ib)}.
 \label{F_N stab}
\end{equation}
Similarly, taking into account~\eqref{S_j ineq},~\eqref{kappa CS}, and~\eqref{q_k stab}, we estimate
\begin{equation} \label{I_j}
 \big\|\hat I_j(\rho^2)\big\|_{L_2({\mathbb R}+ib)} \le  A_R\big\| \hat{\mathrm M}\big\|_{L_2({\mathbb R}+ib)}, \quad j=\overline{2, N}.
\end{equation}
Inequalities~\eqref{|M_1|},~\eqref{F_N stab}, and~\eqref{I_j} yield~\eqref{M_1 stab}.
The theorem is proved.
\end{proof}

\begin{remark}
In Theorems~\ref{main th} and~\ref{th dir}, instead of $\| \cdot \|_{L_2({\mathbb R} + ib)},$ one can consider $\| \cdot \|_{L_2({\mathbb R} + i\tau)}$ with any $\tau > b.$ Then, the constants $A_R$ in the resulting estimates will depend not only on $R,$ but also on $\tau.$
\end{remark}

\section{Stability of recovering from the characteristic functions}
\label{char}

In Theorem~\ref{main th}, the line ${\mathbb{R}}+ib,$ on which the spectral characteristics of two boundary value problems are compared, depends on the parameter $R.$
We can proceed to the comparison on the line $\mathbb R$ if we take other spectral characteristics: the characteristic functions $\Delta(\rho^2)$ and $\Delta_k(\rho^2),$ $k=\overline{2, N}.$

\begin{IP} \label{sec ip}
Given the characteristic functions $\Delta(\lambda)$ and $\Delta_k(\lambda),$ $k=\overline{2, N},$ recover $\mathrm{q}.$
\end{IP}
In view of formulas~\eqref{M_k}, this inverse problem reduces to Inverse problem~\ref{main ip}. Using Theorem~\ref{main th}, it is easy to obtain a theorem on the uniform stability of Inverse problem~\ref{sec ip}.

\begin{theorem} \label{444}
Let $R>0.$
For any potentials $\mathrm{ q}$ and $\mathrm{\tilde q}$ satisfying~\eqref{omega eq} and~\eqref{q bounded} the following estimate holds:
\begin{equation*}
\| \mathrm{\hat q}\| \le A_R \varepsilon,
\end{equation*}
where
$$\varepsilon := \left\| \rho^N \hat \Delta(\rho^2)\right\|_{L_2(\mathbb R)} + \sum_{k=2}^N \left\| \rho^{N-1}\hat \Delta_k(\rho^2)\right\|_{L_2(\mathbb R)}$$
and
$A_R>0$ depends only on $R.$
\end{theorem}
In the paper~\cite{tree}, the uniform stability of Inverse problem~\ref{sec ip} was proved for potentials with the components from $W^{-1}_2.$
Here, we consider another class of potentials with the components from $L_2,$ and the estimates are obtained in stronger norms.
\begin{proof}
In view of Theorem~\ref{main th}, it suffices to prove that for each $k=\overline{2, N}$
\begin{equation*}
\big\| \hat M_k(\rho^2) \big\|_{L_2({\mathbb R}+ib)} \le A_R \varepsilon.
\end{equation*}
Using the results of Lemma~\ref{S_j lemma} and~\eqref{Delta}, we obtain that
$$|\rho^{N}\hat \Delta(\rho^2)| \le A_R e^{\tau N \pi}, \quad  \rho^{N}\hat \Delta(\rho^2) \in L_2(\mathbb R).$$
The function $\rho^{N}\hat \Delta(\rho^2)$ is entire in $\rho.$
By the Paley--Wiener theorem~\cite[\S19.3]{rudin}, the following representation holds:
$$\rho^{N}\hat \Delta(\rho^2) = \int_{-N\pi}^{N\pi} e^{i \rho t} f(t) \, dt, \quad
f\in L_2(-N\pi, N\pi).$$
It yields that
\begin{equation} \big\|\rho^{N}\hat \Delta(\rho^2)\big\|_{L_2({\mathbb R}+ib)} \le A_R \big\|\rho^{N}\hat \Delta(\rho^2) \big\|_{L_2({\mathbb R})}.
\label{f1}
\end{equation}
Similarly, we obtain 
\begin{equation} \label{f2}
\big\|\rho^{N-1}\hat \Delta_k(\rho^2)\big\|_{L_2({\mathbb R}+ib)} \le A_R \|\rho^{N-1}\hat \Delta_k(\rho^2) \|_{L_2({\mathbb R})}.
\end{equation}

From formula~\eqref{M_k} it follows that for $\rho \in {\mathbb R}+ib,$
\begin{equation*}
\begin{split}
\big|\hat M_k(\rho^2)\big| \le \frac{\big| \rho^{N-1} \hat \Delta_k(\rho^2)\big|}{\big|\rho^{N-1}\Delta(\rho^2)\big|}
+ \frac{\big|\rho^{N-2} \tilde \Delta_k(\rho^2)\big|}{\big|\rho^{N-1}\Delta(\rho^2)\big|}
\frac{ \big| \rho^{N} \hat \Delta(\rho^2)\big|}{\big|\rho^{N-1}\tilde\Delta(\rho^2)\big|}
 \\
 \overset{\eqref{Delta est}, \eqref{Delta_k est}}{\le}  A_R \big| \rho^{N-1} \hat \Delta_k(\rho^2)\big| + A_R  \big| \rho^{N} \hat \Delta(\rho^2)\big|.
\end{split}
\end{equation*}
From here, using~\eqref{f1} and~\eqref{f2}, we obtain
$$\big\|\hat M_k(\rho^2)\big\|_{L_2({\mathbb R}+ib)} \le A_R \big\|\rho^{N}\hat \Delta(\rho^2) \big\|_{L_2({\mathbb R})}+
A_R \|\rho^{N-1}\hat \Delta_k(\rho^2) \|_{L_2({\mathbb R})} \le A_R \varepsilon,$$
and the theorem is proved.
\end{proof}

\vspace*{5mm}

\noindent{\it Acknowledgements.} This work was supported by grant No. 24-71-10003 of the Russian Science Foundation. The author thanks N.P.~Bondarenko for the valuable recommendations which helped to improve the paper.

\section*{Appendix: stability of partial derivatives of TOKs}
\label{App}

Consider the Sturm--Liouville equation with a complex-valued potential $q \in L_2(0, \pi):$
\begin{equation*} 
-y''(x) + q(x) y(x) = \lambda y(x), \quad x \in (0, \pi).
\end{equation*}
Let $S(x, \lambda)$ and $C(x, \lambda)$ be its solutions under the initial conditions
$$S(0, \lambda) = C'(0, \lambda) = 0, \quad S'(0, \lambda) = C(0, \lambda) = 1.$$
Here and below, the prime denotes the derivative with respect to the {\it first} argument.
The derivative of a function $f$ with respect to the {\it second} argument is denoted by $\dot{f}.$

We introduce a parameter $\rho$ such that $\lambda = \rho^2$ and $\arg \rho \in [0, \pi).$
For each $x \in [0, \pi],$ the following representations hold (see, e.g.,~\cite{march,freil}):
\begin{equation} \label{trans1}
\begin{split}
S(x, \lambda) = \frac{\sin \rho x}{\rho} +\int_0^x P(x, t) \frac{\sin \rho t}{\rho} \, dt,
\\
C(x, \lambda) =\cos \rho x +\int_0^x K(x, t) \cos \rho t \, dt,
\end{split}
\end{equation}
where
\begin{equation*}
\begin{split}
P, K \in L_2({\cal D}), \quad {\cal D} := \{ (x, t) \colon 0 \le t \le x \le \pi \};\\
P(x, \cdot), K(x, \cdot) \in L_2(0, x), \quad x \in (0, \pi].
\end{split}
\end{equation*}
The functions $P(x, t)$ and $K(x, t)$ are called TOKs (transmutation operator kernels), see~\cite{sitnik, march, freil}.

The smoothness of the TOKs exceeds the smoothness of the potential $q(x)$ by one, see~\cite{freil, march}. In the case $q \in L_2(0, \pi),$ this means that $P, K \in AC({\cal D}),$ and their partial derivatives satisfy the conditions
\begin{equation*}
\begin{split}
P', \; K', \; \dot{P}, \; \dot{K} \in L_2({\cal D});& \\
P'(x, \cdot), \; K'(x, \cdot), \; \dot{P}(x, \cdot), \; \dot{K}(x, \cdot) \in L_2(0, x),& \quad x \in (0, \pi].
\end{split}
\end{equation*}
Besides,
\begin{equation*}
P(x, x) = K(x, x) = \frac12Q(x), \quad Q(x) := \int_0^x q(t) \, dt, \quad P(x, 0) = 0.
\end{equation*}

Differentiating and integrating by parts in~\eqref{trans1}, we obtain
\begin{equation} \label{0}
\begin{split}
S(x, \lambda) = \frac{\sin \rho x}{\rho} - \frac12Q(x) \frac{\cos \rho x}{\rho^2} +\int_0^x \dot P(x, t) \frac{\cos \rho t}{\rho^2} \, dt,
\\
S'(x, \lambda) = \cos \rho x + \frac12Q(x) \frac{\sin \rho x}{\rho} + \int_0^x  P'(x, t) \frac{\sin \rho t}{\rho} \, dt,
\\
C(x, \lambda) =\cos \rho x + \frac12Q(x) \frac{\sin \rho x}{\rho} - \int_0^x \dot K(x, t) \frac{\cos \rho t}{\rho^2} \, dt,
\\
C'(x, \lambda) =-\rho\sin \rho x + \frac12Q(x) \cos \rho x + \int_0^x  K'(x, t) \cos \rho t \, dt.
\end{split}
\end{equation}

We prove the uniform stability of the partial derivatives of $P(x, t)$ and $K(x, t)$  by the potential $q.$
Previously, in the works~\cite{march,hrynivIP}, some results were obtained implying the stability of the TOKs by the spectral characteristics, but not by the potential.
In the paper~\cite{hryniv}, for some auxiliary objects, Lipschitz estimates were obtained in the singular case $q \in W^{-1}_2(0, \pi).$
The uniform stability of the TOKs by the potential in the singular case was proved in~\cite{tree}.  
To the author's knowledge, the stability of the {\it partial derivatives} of the TOKs by $q \in L_2(0, \pi)$ has not been formulated as an independent result. In the paper~\cite{but diff}, the TOK for the integro-differential operator was investigated and an inequality was obtained implying the uniform stability of the partial derivative of $P(x, t)$ with respect to $t,$ see (72) in~\cite{but diff}. 
Here, the result will be formulated for the partial derivatives with respect to $x$ and $t$ of both kernels $P(x, t)$ and $K(x, t).$ Unlike the previous works, we will indicate a specific constant $B_R$ in the stability estimates, which requires a detailed proof.

Consider the second potential $\tilde q \in L_2(0, \pi).$ We agree that if a symbol $\alpha$ denotes some object corresponding to $q$ then $\tilde \alpha$ denotes the similar object corresponding to $\tilde q.$
For brevity, we put $\hat \alpha = \alpha - \tilde \alpha.$

\begin{theorem} \label{trans stab}
Let $R>0.$ Then, for any potentials $q$ and $\tilde q$ satisfying the condition
$$\| q\|_{L_2(0, \pi)} \le R, \quad \| \tilde q \|_{L_2(0, \pi)}\le R,$$
 the following estimates hold:
\begin{equation} \label{K stability}
\begin{split} 
\| \hat K'(x, \cdot)\|_{L_2(0, x)} \le B_R \| \hat q\|_{L_2(0, \pi)}, \quad \| \hat{\dot{K}}(x, \cdot)\|_{L_2(0, x)} \le B_R \| \hat q\|_{L_2(0, \pi)},
\\ 
\| \hat P'(x, \cdot)\|_{L_2(0, x)} \le B_R \| \hat q\|_{L_2(0, \pi)}, \quad \| \hat{\dot {P}}(x, \cdot)\|_{L_2(0, x)} \le B_R \| \hat q\|_{L_2(0, \pi)}, 
\end{split} 
\end{equation}
where $x \in (0, \pi]$ and $B_R = \frac{\sqrt 2}{2} + 12R \pi^{\frac32}\exp(2\pi^{\frac32}R).$
\end{theorem}

First, we provide several auxiliary statements necessary for proving Theorem~\ref{trans stab}.

\begin{proposition}[see~\cite{freil}]
The TOKs can be obtained by the method of successive approxi\-mations:
\begin{equation} \label{series}
P(x,t) = \sum_{n=1}^\infty P_n(x, t), \quad K(x, t) = \sum_{n=1}^\infty K_n(x, t),
\quad 0 \le t \le x \le \pi,
\end{equation}
where
\begin{equation} \label{K_1}
P_1(x,t) = \frac{1}{2}Q\Big(\frac{x+t}{2}\Big) -
\frac{1}{2}Q\Big(\frac{x-t}{2}\Big),
\quad
K_1(x,t)=\frac{1}{2}Q\Big(\frac{x+t}{2}\Big)+\frac{1}{2}
Q\Big(\frac{x-t}{2}\Big),
\end{equation}
and for $n \ge 1,$ 
\begin{equation} \label{K_{n+1}}
\begin{split}
K_{n+1}(x,t)=\frac{1}{2}\int_t^x
\Bigg( \int_{\xi-t}^\xi q(\tau)K_n(\tau,t+\tau-\xi)\,d\tau +\int_{\frac{\xi+t}{2}}^\xi q(\tau)K_n(\tau,t-\tau+\xi)\,d\tau
\\
+\int_{\frac{\xi-t}{2}}^{\xi-t}
q(\tau)K_n(\tau,-t-\tau+\xi)\,d\tau \Bigg)\,d\xi,
\\
P_{n+1}(x,t)=\frac{1}{2}\int_t^x
\Bigg( \int_{\xi-t}^\xi q(\tau)P_n(\tau,t+\tau-\xi)\,d\tau +\int_{\frac{\xi+t}{2}}^\xi q(\tau)P_n(\tau,t-\tau+\xi)\,d\tau
\\
-
\int_{\frac{\xi-t}{2}}^{\xi-t}q(\tau)
P_n(\tau,-t-\tau+\xi)\,d\tau \Bigg)\,d\xi.
\end{split}
\end{equation}
The functions $P_{n}(x, t)$ and $K_{n}(x, t),$ $n \ge 1,$ are continuous on ${\cal D}.$
The series in~\eqref{series} converge uniformly by virtue of the estimates
\begin{equation}\label{K_n}
|P_n(x,t)|, \; |K_n(x,t)|\le \frac{x^{n-1}}{(n-1)!} Q^n_a(x), \quad n \ge 1, \quad Q_a(x) := \int_0^x |q(\xi)|\,d\xi.
\end{equation}
\end{proposition}

Note that the terms in formulas~\eqref{K_1}--\eqref{K_{n+1}} for $P_j$ and $K_j$ are similar while only the signs differ, which are insignificant when estimating the absolute values. We will present further estimates and calculations only for $K(x, t),$ since for $P(x, t)$ they are similar.

\begin{proposition}[the Leibniz formula]
\label{leib}
Consider
$$F(x) = \int_{a(x)}^{b(x)} f(x, y) \, dy,$$
where $a(x) \le b(x)$ and $f(x, \cdot) \in L[a(x), b(x)]$ for all $x \in [0, T].$
Suppose that
\begin{enumerate}
\item[1)] the functions $a(x)$ and $b(x)$ are differentiable and monotone,
\item[2)] for all $x \in [0, T]$ and a.e. $y \in [a(x), b(x)],$ there exists $f'(x,y)$ such that
$$f' \in L({\cal S}), \quad {\cal S} := \big\{ (x, y) \colon x \in [0, T], \; y \in [a(x), b(x)]\big\},$$
\item[3)] the functions $f(x, a(x))$ and $f(x, b(x))$ are defined a.e. on $[0, T]$ and
$$b'(x) f(x, b(x)), \; a'(x)f(x, a(x)) \in L(0, T).$$
\end{enumerate}
Then, $F \in AC[0, T]$ and
\begin{equation} \label{F'}
F'(x) = b'(x) f(x, b(x)) - a'(x) f(x, a(x)) + \int_{a(x)}^{b(x)} f'(x, y) \, dy.
\end{equation}
\end{proposition}
To prove the proposition, one has to integrate the right-hand side of \eqref{F'} from $0$ to $x_0$ and apply the Fubini--Tonelli theorem. Using the Newton--Leibniz formula and changing the variable of integration, we obtain $F(x_0)$ up to a constant term. This means that $F$ is a primitive of the right-hand side of~\eqref{F'}, and formula~\eqref{F'} is valid.

\begin{lemma}
For $n \ge 1$, there exist $K_n'(x, t)$ and $\dot K_n(x, t),$ which can be obtained by the formulas
\begin{equation} \label{first der}
\begin{split}
K_1'(x,t) = \frac{1}{4} q\Big(\frac{x+t}{2}\Big) + \frac{1}{4}q\Big(\frac{x-t}{2}\Big),
\quad {\dot K}_1(x, t) = \frac14 q\Big(\frac{x+t}{2}\Big) - \frac14 q\Big(\frac{x-t}{2}\Big),
\\[2mm]
K_{n+1}'(x, t) = \frac12 \int_{x-t}^x q(\tau) K_n(\tau, t + \tau - x) \,d\tau + \frac12 \int_{\frac{x+t}{2}}^x q(\tau) K_n(\tau, t-\tau+x)\,d\tau \\
+\frac12 \int_{\frac{x-t}{2}}^{x-t} q(\tau) K_n(\tau,x-t-\tau)\,d\tau,
\\
{\dot K}_{n+1}(x, t) = -\frac{1}{2} \int_0^t q(\tau) K_n(\tau, \tau) \, d\tau +\frac12 \int_t^x \Bigg(\int_{\xi-t}^\xi q(\tau){\dot K}_n(\tau, t+\tau-\xi)\,d\tau \\
+\int_{\frac{\xi+t}{2}}^\xi q(\tau){\dot K}_n(\tau, t-\tau+\xi)\,d\tau 
- \int_{\frac{\xi-t}{2}}^{\xi-t} q(\tau){\dot K}_n(\tau, -t-\tau+\xi)\,d\tau\Bigg)\,d\xi \\
+ \frac{1}{4} \int_t^x q\Big(\frac{\xi-t}{2}\Big) K_n\Big(\frac{\xi-t}{2},\frac{\xi-t}{2}\Big) \,d\xi
- \frac{1}{4} \int_t^x q\Big(\frac{\xi+t}{2}\Big) K_n\Big(\frac{\xi+t}{2},\frac{\xi+t}{2}\Big)\, d\xi.
\end{split}
\end{equation}
For $n \ge 2$, the functions $K_n'(x, t)$ and $\dot{K}_n(x, t)$ are continuous on $\cal D$ and satisfy the inequalities
\begin{equation} \label{K' est}
|K_{n}'(x, t)| \le \frac{x^{n-1}}{n!}Q_a^{n+1}(x), \quad |\dot K_{n}(x, t)| \le 2 \frac{x^{n-2}}{(n-1)!}Q^{n}_a(x), \quad 0 \le t \le x \le \pi.
\end{equation}
\end{lemma}
\begin{proof}
I. The formulas for $K_1',$ $\dot K_1,$ and $K'_{n+1}$ in~\eqref{first der} easily follow from~\eqref{K_1} and~\eqref{K_{n+1}} if we consider differentiation as the inverse operation to the Lebesgue integral.
The variable with respect to which differentiation is performed is present only in the upper limits of the integrals.
The continuity of $K_n'(x, t)$ for $n \ge 2$ follows by induction from formula~\eqref{first der}, the absolute continuity of the Lebesgue integral, and  the continuity of $K_{n-1}.$

II. We prove by induction the formula for $\dot K_{n+1}$ in~\eqref{first der} and the continuity of this function.
Put $g(x,t) := q(x)K_n(x, t).$
Assume that the existence of $\dot K_n(x, t)$ and the property
$$ \dot g(x,t) = q(x) \dot K_n(x, t) \in L({\cal D})$$
are proved.
For $n=1,$ the property is verified directly, while for $n \ge 2,$ it follows from the continuity of the function $\dot K_{n}.$

Fix $\xi \in [0, \pi]$ and consider
$$F(\xi, t) = \int_{\xi-t}^\xi g(\tau,t+\tau-\xi)\,d\tau
+\int_{\frac{\xi+t}{2}}^\xi g(\tau,t-\tau+\xi)\,d\tau
+\int_{\frac{\xi-t}{2}}^{\xi-t} g(\tau,-t-\tau+\xi)\,d\tau,$$
where $t \in [0, \xi].$
Each of the integrals in this expression satisfies the conditions of Proposition~\ref{leib}. Therefore, there exists $\dot F(\xi, \cdot) \in L(0, \xi)$ and 
\begin{equation} \label{dot F}
\begin{split}
\dot F(\xi, t) = -\frac12 g\Big(\frac{\xi+t}{2},\frac{\xi+t}{2}\Big)+\frac12g\Big(\frac{\xi-t}{2},\frac{\xi-t}{2}\Big)\\
+\int_{\xi-t}^\xi \dot g(\tau,t+\tau-\xi)\,d\tau
+\int_{\frac{\xi+t}{2}}^\xi \dot g(\tau,t-\tau+\xi)\,d\tau
-\int_{\frac{\xi-t}{2}}^{\xi-t} \dot g(\tau,-t-\tau+\xi)\,d\tau.
\end{split}
\end{equation}
Using this formula and~\eqref{K_n}, after changing the order of integration, we obtain
$$ \|\dot F(\xi, \cdot) \|_{L(0, \xi)} \le \int_0^\xi |q(v)| |K_n(v,v)| \, dv +
3 \| \dot g\|_{L({\cal D})} \le \frac{\pi^{n-1}}{(n-1)!} Q_a^{n+1}(\pi) + 3 \| \dot g\|_{L({\cal D})},$$
wherein the estimate does not depend on $\xi.$ Therefore, $\dot F \in L({\cal D}).$

Then, the function
$K_{n+1}(x, t) = \frac12 \int_t^x F(\xi, t) \, d\xi$ satisfies the conditions of Proposition~\ref{leib}, and there exists
$$\dot K_{n+1}(x, t) = -\frac12F(t, t) + \frac12\int_t^x \dot F(\xi, t) \, d\xi,$$
which together with~\eqref{dot F} leads to the formula for $\dot K_{n+1}$ in~\eqref{first der}.  In this formula, the order of integration can be changed and the integration variables can be replaced so that $x$ and $t$ are contained only within the integrals limits. Then, the continuity of $\dot K_{n+1}$ follows from the induction hypothesis and the absolute continuity of the Lebesgue integral.

III. It remains to obtain~\eqref{K' est}. We provide the proof of the estimates for $\dot K_n,$ because the proof of the estimates for $ K'_n$ is simpler. From~\eqref{first der}, we get
\begin{equation} \label{comp} 
\begin{split}
\dot K_{n+1}(x, t) = -\frac{1}{2} \int_0^{\frac{x+t}{2}} q(\tau) K_n(\tau, \tau) \,d\tau +\frac{1}{2} \int_0^{\frac{x-t}{2}} q(\tau) K_n(\tau, \tau) \,d\tau \\
+\frac12\int_t^x \Bigg(\int_{\xi-t}^\xi q(\tau)\dot K_n(\tau, t+\tau-\xi)\,d\tau +
\int_{\frac{\xi+t}{2}}^\xi q(\tau)\dot K_n(\tau, t-\tau+\xi)\,d\tau \\
-\int_{\frac{\xi-t}{2}}^{\xi-t} q(\tau)\dot K_n(\tau, -t+\tau+\xi)\,d\tau\Bigg)d\xi.
\end{split}
\end{equation}
Substituting~\eqref{K_1} and the formula for $\dot{K}_1$ from~\eqref{first der}, we have
\begin{equation*}
\begin{split}
|\dot K_2(x, t)| \le \int_0^x Q_a(\tau) |q(\tau)|\, d\tau
+ \frac18 \int_t^x \left(\int_{\frac{\xi-t}{2}}^\xi |q(\tau)| \Big|q\Big(\tau - \frac{\xi-t}{2}\Big)\Big|\,d\tau \right.\\
\left.+\int_{\frac{\xi-t}{2}}^\xi |q(\tau)| \Big|q\Big(\frac{\xi-t}{2}\Big)\Big|\,d\tau
+\int_{\frac{\xi+t}{2}}^\xi |q(\tau)| \Big|q\Big(\tau - \frac{\xi+t}{2}\Big)\Big|\,d\tau
+\int_{\frac{\xi+t}{2}}^\xi |q(\tau)| \Big|q\Big(\frac{\xi+t}{2}\Big)\Big|\,d\tau\right)\,d\xi.
\end{split}
\end{equation*}
Changing the order of integration, taking into account the monotonicity of $Q_a(\tau),$ we obtain
$|\dot K_2(x, t)| < 2 Q^2_a(x).$ Thus,~\eqref{K' est} is proved for $n=2.$

Further, we prove by induction. Let~\eqref{K' est} hold for some $n \ge 2.$ From~\eqref{K_n}, \eqref{K' est}, and~\eqref{comp}, it follows that
\begin{equation*}
\begin{split}
|\dot{K}_{n+1}(x,t)| \le \int_0^x |q(\tau)| \frac{\tau^{n-1}}{(n-1)!} Q_a^n(\tau) \,d\tau + 2\int_0^x \int_0^\xi |q(\tau)| \frac{\tau^{n-2}}{(n-1)!} Q_a^n(\tau)\,d\tau\,d\xi \\
\le\frac{x^{n-1}}{(n-1)!} \int_0^x Q^n_a(\tau) |q(\tau)|\, d\tau + 2\int_0^x \frac{\xi^{n-2}}{(n-1)!} \,d\xi \int_0^x Q^n_a(\tau) |q(\tau)|\, d\tau\\
= \frac{x^{n-1}}{(n-1)!} Q_a^{n+1}(x) \left( \frac{1}{n+1} + \frac{2}{(n+1)(n-1)}\right) = \frac{1}{n-1}\frac{x^{n-1}}{(n-1)!} Q_a^{n+1}(x).
\end{split}
\end{equation*}
Since $\frac{1}{n-1} \le \frac2n$ for $n \ge 2,$ estimate~\eqref{K' est} holds for $|\dot{K}_{n+1}(x,t)|.$ The lemma is proved.
\end{proof}
From this lemma and equality~\eqref{series}, we obtain a corollary.
\begin{corollary}
The following representations hold:
\begin{equation} \label{der series}
K'(x, t) = K'_1(x,t) + \sum_{n=2}^\infty K'_{n}(x, t), \quad \dot K(x, t) = \dot K_1(x,t) + \sum_{n=2}^\infty \dot K_{n}(x, t),
\end{equation}
where the series converge absolutely and uniformly.
\end{corollary}

Since $K'_n$ and $\dot K_n$ are continuous for $n \ge 2,$ the series in~\eqref{der series} converge to continuous functions.
The property $K'(x, \cdot), \dot{K}(x, \cdot) \in L_2(0, x)$ for each $x \in (0, \pi]$ is conditioned by the terms $K'_1$ and $\dot K_1.$ A similar effect arises when constructing the TOKs in the singular case $q \in W^{-1}_2(0, \pi),$ see~\cite{addendum}.

From formulas~\eqref{K_1},~\eqref{K_{n+1}},~\eqref{first der}, and~\eqref{der series}, we obtain the following proposition.
\begin{proposition} \label{K'(x, x)}
For a.e. $x \in (0, \pi),$ we have
$$ \dot K(x, 0) = 0, \quad \dot K(x, x) = \frac{q(x) - q(0)}{4} - \frac{Q^2(x)}{8}, \quad K'(x, x) = \frac{q(x) + q(0)}{4} + \frac{Q^2(x)}{8}.$$
Similarly, for a.e. $x \in (0, \pi),$ we have
$$ \dot P(x, x) = \frac{q(x) + q(0)}{4} - \frac{Q^2(x)}{8}, \quad P'(x, x) = \frac{q(x) - q(0)}{4} + \frac{Q^2(x)}{8} \quad P'(x, 0) = 0. $$
\end{proposition}

Now, we provide the proof of Theorem~\ref{trans stab}.
\begin{proof}[Proof of Theorem~\ref{trans stab}]
I. 
From formula~\eqref{K_1}, we obtain the inequality
\begin{equation*} \label{ind base}
|\hat K_1(x, t)| \le \hat Q_a(x) \le \sqrt{x} \| \hat q\|_{L_2(0, x)}, \quad \hat Q_a(x) := \int_0^x |\hat q(\tau)|\,d\tau.
\end{equation*}
From~\eqref{K_{n+1}}, it follows that
\begin{equation*}
\begin{split}
|\hat K_{n+1}(x, t)| \le \frac12 \int_t^x\Bigg( \int_{\xi-t}^\xi |\hat q(\tau)| |K_n(\tau,t+\tau-\xi)|\,d\tau +\int_{\frac{\xi+t}{2}}^\xi |\hat q(\tau)| |K_n(\tau,t-\tau+\xi)|\,d\tau \\
+\int_{\frac{\xi-t}{2}}^{\xi-t} |\hat q(\tau)| |K_n(\tau,-t-\tau+\xi)|\,d\tau +  \int_{\xi-t}^\xi |\tilde q(\tau)| |\hat K_n(\tau,t+\tau-\xi)|\,d\tau \\
+\int_{\frac{\xi+t}{2}}^\xi |\tilde q(\tau)||\hat K_n(\tau,t-\tau+\xi)|\,d\tau + \int_{\frac{\xi-t}{2}}^{\xi-t} |\tilde q(\tau)||\hat K_n(\tau,-t-\tau+\xi)|\,d\tau\Bigg) \,d\xi,
\end{split}
\end{equation*}
where $n \ge 1.$
Using these inequalities and~\eqref{K_n}, by induction, we get the estimates
\begin{equation} \label{hat K_n}
|\hat K_n(x, t)|\le 2 \hat Q_a(x) \frac{Q_{m}^{n-1}(x) x^{n-1}}{(n-1)!}, \; n \ge 1, \quad Q_{m}(x) := \int_0^x \max\{ |q(\xi)|, |\tilde q(\xi)|\}\, d\xi.
\end{equation}

II. Using~\eqref{K_n},~\eqref{first der}, and~\eqref{hat K_n}, we obtain
$$|\hat K'_{n+1}(x, t)| \le 3\hat Q_a(x) \frac{Q^{n}_{m}(x) x^{n-1}}{(n-1)!}, \quad n \ge 1.$$
From the formula for $K'_1$ in~\eqref{first der}, it follows that
$$\| \hat K'_1(x, \cdot) \|_{L_2(0,x)} \le \frac{\sqrt2}{2}\| \hat q\|_{L_2(0, x)}.$$
Using these inequalities and~\eqref{der series}, we estimate
\begin{equation*}\begin{split}
\|\hat K'(x, \cdot)\|_{L_2(0, x)} \le \big\| \hat K'_1(x, \cdot) \big\|_{L_2(0,x)} + \sum_{n=1}^\infty \big\| \hat K'_{n+1}(x, \cdot) \big\|_{L_2(0,x)} \\
\le \frac{\sqrt2}{2}\| \hat q\|_{L_2(0, x)}
+ 3\sqrt x \sum_{n=1}^\infty \hat Q_a(x) Q^n_{m}(x) \frac{x^{n-1}}{(n-1)!}\\
= \frac{\sqrt2}{2}\| \hat q\|_{L_2(0, x)} +
3\sqrt x \hat Q_a(x) Q_{m}(x) e^{x Q_m(x)}.
\end{split}
\end{equation*}
Taking into account the fact that
\begin{equation} \label{Q_a}
\hat Q_a(x) \le \sqrt{x} \| \hat q \|_{L_2(0, x)}, \quad Q_{m}(x) \le \sqrt x (\| q\|_{L_2(0, x)} + \| \tilde q\|_{L_2(0, x)}) \le 2 R\sqrt x,
\end{equation}
we arrive at the first inequality in~\eqref{K stability}.

III. From~\eqref{first der}, it follows that
\begin{equation} \label{final est}
\begin{split}
|\hat {\dot K}_{n+1}(x, t)|  \le  \int_0^{x} |\hat q(\tau)| |K_n(\tau, \tau)| \,d\tau 
+ \int_0^{x} |\tilde q(\tau)| |\hat K_n(\tau, \tau)| \,d\tau 
\\
+\frac12\int_t^x \Bigg(\int_{\xi-t}^\xi |\hat q(\tau)||\dot K_n(\tau, t+\tau-\xi)|\,d\tau +
\int_{\frac{\xi+t}{2}}^\xi |\hat q(\tau)||\dot K_n(\tau, t-\tau+\xi)|\,d\tau \\
+\int_{\frac{\xi-t}{2}}^{\xi-t} |\hat q(\tau)||\dot K_n(\tau, -t-\tau+\xi)|\,d\tau\Bigg)d\xi 
+\frac12\int_t^x \Bigg(\int_{\xi-t}^\xi |\tilde q(\tau)|\big|\hat{\dot K}_n(\tau, t+\tau-\xi)\big|\,d\tau \\
+\int_{\frac{\xi+t}{2}}^\xi |\tilde q(\tau)||\hat{\dot K}_n(\tau, t-\tau+\xi)|\,d\tau+\int_{\frac{\xi-t}{2}}^{\xi-t} |\tilde q(\tau)||\hat{\dot K}_n(\tau, -t-\tau+\xi)|\,d\tau\Bigg)d\xi.
\end{split}
\end{equation}
Let $n=1.$
Using the formula for $\dot K_1$ in~\eqref{first der}, applying~\eqref{K_n} and~\eqref{hat K_n}, we obtain
\begin{equation*} \label{base}
|\hat{\dot K}_2(x, t)| \le 6 \hat Q_a(x) Q_{m}(x).
\end{equation*}
We prove by induction the estimate
\begin{equation} \label{*}
|\hat{\dot K}_{n+1}(x, t)| \le 6 \hat{Q}_a(x) Q_{m}^{n}(x) \frac{x^{n-1}}{(n-1)!}, \quad n \ge 1.
\end{equation}
For $n=1,$ the formula is proved. For $n > 1,$ we estimate the right-hand side of~\eqref{final est} using~\eqref{K_n}, \eqref{K' est}, \eqref{hat K_n}, and~\eqref{*}. We obtain
$$ |\hat{\dot K}_{n+1}(x, t)| \le Q_{m}^n(x) \hat{Q}_a(x) \frac{x^{n-1}}{(n-1)!} \left(1 + \frac{10}{n}\right),$$
wherein for $n \ge 2,$ the value $1 + \frac{10}{n}$ does not exceed $6.$ Then, estimate~\eqref{*} is proved by induction. 

From the formula for $\dot{K}_1$ in~\eqref{first der},  it follows that
$$\| \hat {\dot K}_1(x, \cdot)\|_{L_2(0, x)} \le \frac{\sqrt 2}{2}\| \hat q\|_{L_2(0, x)}.$$
Applying~\eqref{der series},~\eqref{*}, and this inequality, we obtain
\begin{equation*}
\begin{split}
\|\hat{\dot K}(x, \cdot)\|_{L_2(0, x)} \le \big\| \hat{\dot K}_1(x, \cdot) \big\|_{L_2(0,x)} + \sum_{n=1}^\infty \big\| \hat{\dot K}_{n+1}(x, \cdot) \big\|_{L_2(0,x)}
\\
\le  \frac{\sqrt2}{2}\| \hat q\|_{L_2(0, x)} + 6\sqrt x \sum_{n=1}^\infty \hat Q_a(x) Q^n_{m}(x) \frac{x^{n-1}}{(n-1)!}\\
= \frac{\sqrt2}{2}\| \hat q\|_{L_2(0, x)} + 6\sqrt xQ_a(x) Q_{m}(x) e^{x Q_m(x)}.
\end{split}
\end{equation*}
Taking into account~\eqref{Q_a}, we arrive at the second inequality in~\eqref{K stability}.
\end{proof}

\end{document}